\documentclass{amsart}
\usepackage{amscd}
\usepackage{amssymb}

\def\Kweb#1{   
http:\linebreak[3]//www.\linebreak[3]math.\linebreak[3]uiuc.%
\linebreak[3]edu/\linebreak[3]{K-theory/#1/}}

\def\onto{\twoheadrightarrow}

\def\k{k} 

\def\cdh{\mathrm{cdh}}
\def\coker{\operatorname{coker}}
\def\nil{\operatorname{nil}}
\def\oo{\otimes}

\def\Pic{\operatorname{Pic}}
\def\Proj{\operatorname{Proj}}
\def\Spec{\operatorname{Spec}}
\def\Sym{\operatorname{Sym}}
\def\tors{\operatorname{tors}}
\def\zar{\mathrm{zar}}
\def\lra{\longrightarrow}
\def\map#1{\ {\buildrel #1 \over \lra}\ }
\def\smap#1{\ {\buildrel #1 \over \to}\ }
\newcommand{\comment}[1]{}	

\newcommand{\bbH}{\mathbb H}
\newcommand{\Hcdh}{H_\cdh}
\def\tK{\widetilde K}
\def\tR{\widetilde R}

\newcommand{\bbP}{\mathbb{P}_k}

\def\cF{\mathcal F}
\def\cO{\mathcal O}
\def\cU{\mathcal U}
\newcommand{\C}{\mathbb{C}}
\newcommand{\Q}{\mathbb{Q}}

\newcommand{\Z}{\mathbb{Z}}
\def\frakc{\mathfrak c}
\def\frakm{\mathfrak m}

\input xypic
\xyoption{all} 
\numberwithin{equation}{section}

\theoremstyle{plain}
\newtheorem{thm}[equation]{Theorem}
\newtheorem*{thm*}{Theorem}
\newtheorem{cor}[equation]{Corollary}
\newtheorem{lem}[equation]{Lemma}
\newtheorem{prop}[equation]{Proposition}

\theoremstyle{definition}

\theoremstyle{remark}
\newtheorem{rem}[equation]{Remark}
\newtheorem{ex}[equation]{Example}
\newtheorem{trick}[equation]{Standard Trick}
\newtheorem{subremark}{Remark}[equation] 
\newtheorem{subex}[subremark]{Example} 

\begin{document}
\bibliographystyle{plain}

\title{
$K$-theory of cones of smooth varieties}

\author{G. Corti\~nas}
\thanks{Corti\~nas' research was supported by CONICET and partially supported by
grants PICT 2006-00836, UBACyT-X051, and MTM2007-64704.}
\address{Dep. Matem\'atica, FCEyN-UBA\\ Ciudad Universitaria Pab 1\\
1428 Buenos Aires, Argentina}
\email{gcorti@dm.uba.ar}\urladdr{http://mate.dm.uba.ar/\~{}gcorti}

\author{C. Haesemeyer}
\address{Dept.\ of Mathematics, UCLA, Los Angeles, CA 90095, USA}
\email{chh@math.ucla.edu}

\author{M.\,E. Walker}
\thanks{Haesemeyer and Walker were partially supported by NSF grants}
\address{Dept.\ of Mathematics,
University of Nebraska--Lincoln \\
Lincoln, NE 68588, USA}
\email{mwalker5@math.unl.edu}

\author{C. Weibel}
\thanks{Weibel was supported by NSA and NSF grants}
\address{Dept.\ of Mathematics, Rutgers University, New Brunswick,
NJ 08901, USA} \email{weibel@math.rutgers.edu}

\date{\today}

\begin{abstract}
Let $R$ be the homogeneous coordinate ring of a smooth projective variety $X$
over a field $\k$ of characteristic~0. We calculate the
$K$-theory of $R$ in terms
of the geometry of the projective embedding of $X$.
In particular, if $X$ is a curve then we calculate $K_0(R)$ and $K_1(R)$, and
prove that $K_{-1}(R)=\oplus H^1(C,\cO(n))$.
The formula for $K_0(R)$ involves the Zariski cohomology of twisted
K\"ahler differentials on the variety.
\end{abstract}

\maketitle

\section*{}

\vspace{-28pt}

Let $R=k\oplus R_1\oplus\cdots$ be the homogeneous coordinate ring of
a smooth projective
variety $X$ over a field $k$ of characteristic~0.
In this paper we compute the lower $K$-theory ($K_i(R)$, $i{\le1}$)
in terms of the Zariski cohomology groups $H^*(X,\cO(t))$ and
$H^*(X,\Omega^*_{X}(t))$, where $\cO(1)$ is the ample line bundle of the
embedding and $\Omega^*_{X}$ denotes the K\"ahler differentials of $X$
relative to $\Q$.
We also obtain computations of the higher $K$-groups
$K_n(R)/K_n(k)$, especially for curves.
A complete calculation for the conic $xy=z^2$ is given in
Theorem \ref{thm:conic}.
These calculations have become possible thanks to the new
techniques introduced in \cite{chsw}, \cite{chw-v} and \cite{chwnk}.

Here, for example, is part of Theorem \ref{thm:main};
$R^+$ is the seminormalization of $R$.

\begin{thm*}\label{intro:thm}
Let $R$ be the homogeneous coordinate ring of a smooth $d$-dimensional
projective variety $X$ in $\bbP^N$. Then $\Pic(R)\cong(R^+/R)$ and
\begin{gather*}
K_0(R) \cong \Z\oplus \Pic(R) \oplus \bigoplus\nolimits_{i=1}^{d}
\bigoplus\nolimits_{t=1}^\infty H^{i}(X,\Omega_{X}^{i}(t)),
\quad\text{and}\\
K_{-m}(R) \cong \bigoplus\nolimits_{i=0}^{d-m}
\bigoplus\nolimits_{t=1}^\infty H^{m+i}(X,\Omega_{X}^{i}(t)),
\quad m>0.
\end{gather*}
We have $K_{-m}(R)=0$ for $m>d$, and
$K_{-d}(R)=\bigoplus_{t\ge1\mathstrut} H^d(X,\cO(t))$.

If $k$ has finite transcendence degree over $\Q$ then $K_0(R)/\Z$ and
each $K_{-m}(R)$ are finite-dimensional $k$-vector spaces.
\end{thm*}

For example, if $X=\Proj(R)$ is a smooth curve over $k$ which is
definable over a number field contained in $k$,
we show that
$\Omega^1_{k}\oo\cO(t)\to\Omega^1_X(t)$ induces:
\begin{equation}\label{intro:K0(R)}
K_0(R)=\Z\oplus\Pic(R)\oplus\left(\Omega^1_{k}\oo K_{-1}(R)\right),\quad
K_{-1}(R)\cong\oplus_{t=1}^\infty H^1(X,\cO_X(t)).
\end{equation}
We also have $K_n^{(n+2)}(R)\cong \Omega^{n+1}_k\oo K_{-1}(R)$
for all $n\ge1$. (See Proposition \ref{Kunneth-Kni}(d).)

When $R$ is normal, \eqref{intro:K0(R)} implies that
$K_0(R)=\Z$ holds if and only if either
(a) $k$ is algebraic over $\Q$, or (b) $K_{-1}(R)=0$.
Case (a) was discovered by Krishna and Srinivas \cite[1.2]{KSr},
while parts of case (b) were discovered in \cite{WeibelNorm}.
By Riemann-Roch, the vanishing of $K_{-1}(R)$ is equivalent to the
vanishing of the vector spaces $H^0(X,\Omega^1_{X/k}(-t))$ for $t>0$,
which is a delicate arithmetic question (unless, for example,
the embedding has degree $d\ge2g-2$).
Note that case (b) clarifies Srinivas' theorem in \cite{Sr1} that
when $k=\C$ and $H^1(X,\cO(1))\ne0$ we have $K_0(R)\ne\Z$.

Still assuming that $X$ is a curve, suppose in addition that
$k$ is a number field; then $\Omega^1_k=0$ and hence
$K_0(R)=\Z\oplus(R^+/R)$. We also establish (in \ref{number-curve} and
\ref{K12}) the previously unknown calculations that
\begin{gather}\label{intro:K1(R)}
K_1(R)=k^\times\! \oplus \left[\bigoplus_{t=1}^\infty
H^0(X,\Omega^1_{X/k}(t))\right]/\Omega^1_{R/k},
\quad
K_2(R) = K_2(k) \oplus \tors\Omega^1_{R/k}, \\ \label{intro:Kn(R)}
K_n(R) = K_n(k) \oplus HC_{n-1}(R)/HC_{n-1}(k), \qquad n\ge3.
\end{gather}
The $K_1$ formula \eqref{intro:K1(R)} is a clarification of a result of
Srinivas \cite{SrCone}. When $k$ is not algebraic over $\Q$, formulas
\eqref{intro:K0(R)}, \eqref{intro:K1(R)} and \eqref{intro:Kn(R)}
need to be altered to involve the arithmetic Gauss-Manin connection;
see Proposition \ref{K02} and Example \ref{ex:partial_vanish}.

For any smooth $d$-dimensional variety $X$, $K_0(R)/\Z$ is the direct sum of
the eigenspaces $K_0^{(i)}(R)$ of the Adams operation, $1\le i\le d+1=\dim R$,
and we give a formula for these eigenspaces.
For example, the top eigenspace, $K_0^{(d+1)}(R)$, may be identified
with the Chow group of smooth zero-cycles in $\Spec(R)$;  we show that
\[ K_0^{(d+1)}(R) \cong
\bigoplus\nolimits_{t=1}^\infty H^{d}(X,\Omega^{d}_{X}(t)). \]

As pointed out in \cite{KSr}, the normal domain $R_k=k[x,y,z]/(x^n+y^n+z^n)$
has $K_0(R_{\Q})=\Z$ but if $n\ge4$ then $H^1(X,\cO(1))$ is nonzero while
$K_0(R_{\C})/\Z$ is a very big $\C$-vector space;
by \eqref{intro:K0(R)}, it is the direct sum of the
$\Omega^1_{\C}\oo H^1(X,\cO(t))$, $t\ge1$.

We also obtain reasonably nice formulas for the eigenspaces
$K_n^{(i)}(R)$ when $n>0$ and $i\ge n$; see Theorem \ref{Kni}.
To illustrate the range of our cohomological results,
consider $K_1(R)$ when $X$ is a smooth curve and $R$ is normal; we have
$K_1(R)= k^\times \oplus K_1^{(2)}(R) \oplus K_1^{(3)}(R)$, where
\begin{gather}
K_1^{(2)}(R)\cong
   \left(\bigoplus\nolimits_{t=1}^\infty
   H^0(X,\Omega^1_{X}(t)\right)/\Omega^1_{R},
\qquad\text{and}\nonumber \\
K_1^{(3)}(R)=\bigoplus_{t=1}^\infty \coker\bigl\{
\Omega^1_{\k}\oo H^0(X,\Omega_{X/k}^1(t))\map{\nabla}
 \Omega^2_{\k}\oo H^1(X,\cO_X(t)) \bigr\}.\label{intro:partial}
\end{gather}
The map $\nabla$ in \eqref{intro:partial} is a twisted
Gauss-Manin connection (see Lemma \ref{lem:FundSeq}).
In Section \ref{sec:curvecones}, we prove that if $n\ge1$ then
$K_n^{(n+1)}(R)$ contains $\Omega^{n-1}_\k\otimes_\Q k^{d+g-1}$
as a direct summand provided that either
\begin{enumerate}
\item[(a)] $X$ has genus $g$ and is embedded in $\bbP^N$ by
a complete linear system of degree $d$, with $d\ge 2g-1$, or

\item[(b)] $X$ is induced by base change to $k$ from a curve defined
over a number field contained in $k$.
\end{enumerate}
(See Theorem \ref{nonvanish} and Example \ref{ex:partial=0}.)
In particular $K_1^{(2)}(R)\ne 0$, and in
general, $K_n^{(n+1)}(R)\ne 0$ if $n-1\le \mathrm{tr}.\deg(\k/\Q)$.
Observe that the case $n=1$ improves the result of Srinivas in
\cite[\S1]{SrCone}
that there is a surjection from $\tK_1(R) = K_1(R)/K_1(k)$ to
$H^0(X,\Omega^1_{X/k}(1))$ and hence that $\tK_1(R)\ne 0$ if $d\ge 2g+1$.

Finally, in Theorem \ref{thm:conic} we give a complete calculation of
the $K$-theory of the homogeneous coordinate ring of the plane conic,
$R=k[x,y,z]/(xy-z^2)$.

This paper is organized as follows.
In Section \ref{sec:standard}, we reduce the calculation of $K_n(R)$
to a $cdh$-cohomology computation and knowledge of $HC_{n-1}(R)$.
This relies on the basic
observation that cones are $\mathbb{A}^1$-contractible, so that the
reduced $K$-theory $\tK_n(R) = K_n(R)/K_n(k)$ can be calculated in
terms of $NK_n(R)$, making our previous calculations (see \cite{chsw},
\cite{chw-v}, \cite{chwnk}) applicable.  Several of the formulas we
obtain are valid for general graded algebras of the form $R=k\oplus
R_1\oplus\cdots$. We also specialize these formulas to the case when
$\dim R=2$, and obtain an expression for $K_n(R)$ in terms of $cdh$
cohomology and cyclic homology $(n\ge 1)$.

In Section \ref{sec:cones}
we compute the $cdh$ terms in the formulas of the previous sections
for the case when $R$ is the affine cone of a smooth variety.
In Section \ref{sec:curvecones}, we return to the case when the graded
coordinate ring has dimension $2$, that is, we investigate cones over
smooth projective curves.  Finally, in Section \ref{sec:conic} we apply
the techniques of this paper to completely determine the $K$-theory of
$R=k[x,y,z]/(xy-z^2)$.

{\it Notations:}
Throughout this paper we consider (commutative, unital) algebras over
a fixed ground field $k$, which we assume has characteristic zero.
Undecorated tensor products $\oo$ and differential forms
$\Omega^*$ are taken over $\Q$; we write $\oo_k$ and $\Omega^*_{/k}$
for tensor product and forms relative to $k$. Similarly, cyclic
homology is always taken over $\Q$.
If $F$ is a functor defined on schemes over $k$, we will write $F(R)$
for $F(\Spec(R))$. If $R$ is an augmented $k$-algebra (for example, the
homogeneous coordinate ring of a variety), and $F$ is a functor from
rings to some abelian category, then we write $\widetilde{F}(R)$ for the
(split) quotient $F(R)/F(k)$.


\section{$K$-theory of graded algebras}\label{sec:standard}

Throughout this section, we
let $R = R_0\oplus R_1\oplus\cdots$ be a finitely generated
graded algebra over a field $k$ of characteristic~0 such that $R_0$
is a local, artinian $k$-algebra whose residue field is isomorphic to
$k$ as a $k$-algebra.
These conditions ensure that the map $K_n(R) \to K_n(k)$ induced
by the composition of $R \onto R_0 \onto k$ is a split surjection. For
example, $R_0$ might be $k$ itself, and indeed for most of the
calculations in this paper, one may as well assume $R_0 = k$.
Let $\frakm_R$ denote the unique graded maximal ideal of $R$; that is,
$\frakm_R$ is the kernel of the split surjection $R \onto k$.

We let $R_{red}$ denote the reduced ring associated to $R$. It is a
graded ring whose degree $0$ piece is the field $k$.  We let $\tR$
denote the normalization of $R_{red}$ (i.e., the integral closure of
$R_{red}$ in its ring of total quotients).  It is well known that
$\tR = \tR_0 \oplus \tR_1 \oplus \cdots$ is graded, that $\tR_0$ is a
product of fields, and that $\Pic(\tR)=0$.

We let $R^+$ denote the semi-normalization of $R_{red}$, that is, the
maximal extension of $R_{red}$ inside its total quotient ring $Q$ such
that for all $x\in Q$, $x^2, x^3\in R^+$ implies $x\in R^+$;
see \cite{Swan}. Alternatively,
$\Spec(R^+) \to \Spec(R_{red})$ is a universal homeomorphism.

We are interested in computing the kernel $\tK_n(R)$
of the split surjection $K_n(R)\to K_n(k)$, for $n=1,0,-1,\dots,1-d$.
(By \cite{chsw}, $K_n(R)=NK_n(R)=0$ for $n\le-d$.)
In general, for any graded ring $R=R_0\oplus R_1\oplus R_2\oplus\dots$,
the groups $\tK_n(R)$ are known to be $R_0$-modules (see \cite{W81}),
and hence (since $R_0$ contains $\Q$) they are
uniquely divisible as abelian groups.  Thus there is a decomposition
$\tK_n(R)\cong \bigoplus_i\tK_n^{(i)}(R)$ according to the
eigenvalues $k^i$ of the Adams operations $\psi^k$.

\begin{rem}\label{rem:R+}
Suppose that the punctured spectrum,
$\Spec(R_{red}) \setminus \{\frakm_R\}$, is non-singular.
Then the conductor $\frakc$ to the normalization $\tR$ of $R_{red}$
is $\frakm_R$-primary. 
An easy calculation shows that the seminormalization of $R_{red}$ is
$$
R^+= k \oplus\tR_1\oplus\tR_2\oplus\cdots,
$$
with $\tR/R^+=\tR_0/k$ and $R^+/R_{red}=\tR/(\tR_0+R_{red})$.
Then $K_n^{(i)}(R)\cong\tK_n^{(i)}(R)$ for $n\le1$,
with two exceptions:
$\tK_0^{(0)}(R)=0$, and $\tK_1^{(1)}(R)\cong \nil(R)/\nil(R_0)$.
The problem of computing $\tR/R_{red}$ (and hence $R^+/R_{red}$) is hard.
\end{rem}

The main results of this section, Theorems \ref{thm:A} and \ref{Kni},
are formulated in terms of the $cdh$ cohomology groups
$\Hcdh^*(R,\Omega^i)$ introduced in \cite{chsw} and \cite{chw-v},
where the K\"ahler differentials, $\Omega^i = \Omega^i_{-/\Q}$, are taken
relative to the base field $\Q$.
By \cite[2.5]{chwnk}, we have that $\Hcdh^0(R,\cO)=R^+$.
For simplicity, we write $\Hcdh^{m}(R,\Omega^{i})/d\Hcdh^{m}(R,\Omega^{i-1})$
for the cokernel of the map
$d:\Hcdh^{m}(R,\Omega^{i-1})\to\Hcdh^m(R,\Omega^i)$
induced by the K\"ahler differential.
Theorem \ref{thm:A} will follow from Proposition \ref{weight1}
and Theorem \ref{thm:hypercohom} below.

\begin{thm}\label{thm:A}
Let $R=R_0\oplus R_1\oplus\cdots$ be a finitely generated 
graded algebra over a field $k$ of characteristic~0. Assume $R_0$ is local
artinian with residue field $k$.
Then the Adams operations induce an eigenspace decomposition:
\[
K_0(R) = \Z \oplus R^+/R_{red} \oplus
\bigoplus_{i=1}^{\dim\,R-1} H_\cdh^i(R,\Omega^i)/d\,H_\cdh^i(R,\Omega^{i-1}).
\]
The negative $K$-groups are given by
\[ K_{-m}(R) = \Hcdh^m(R,\cO) \oplus \bigoplus_{i=1}^{\dim\,R-m-1}
H_\cdh^{m+i}(R,\Omega^i)/d\,H_\cdh^{m+i}(R,\Omega^{i-1}).
\]
for $m>0$.
Here, $K_0^{(0)}(R) = \Z, K_0^{(1)}(R) = R^+/R_{red},
K_{-m}^{(1)}(R) = \Hcdh^m(R,\cO)$ and
the groups indexed by $i$ are $K_0^{(i+1)}$ and
$K_{-m}^{(i+1)}(R)$, respectively.
\end{thm}

By \cite[1.2 and 2.3]{WKH} we have $KH_*(R) \cong KH_*(R_0) \cong
K_*(k)$, and
thus by \cite[1.6]{chw-v}, we have
\begin{equation}\label{eq:tKq}
\tK_n(R) \cong \pi_n \cF_K(R) \cong \pi_{n-1}\cF_{HC}(R)
\quad \text{ for all $n$.}
\end{equation}
Here, $\cF_{HC}(R)=\cF_{HC}(R/\Q)$ is the homotopy fiber of
$HC(R)\to\bbH_\cdh(R,HC)$, with cyclic homology taken relative to
the subfield $\Q$ of $\k$, so that there is a long exact sequence
\begin{equation*}
\cdots \to HC_n(R) \to \bbH_\cdh^{-n}(R,HC)\to \tK_n(R)
	\to HC_{n-1}(R) \to \cdots.
\end{equation*}
These groups all have $\lambda$-decompositions and the maps in this
sequence are compatible with these decompositions (see \cite{chw-chern}),
but there is a weight shift in that
$\tK_n^{(i)}(R)$ maps to $HC^{(i-1)}_{n-1}(R)$.
We have $\tK_n^{(0)}(R)=0$ for all $n$ because $\cF_{HC}^{(-1)}(R)
\simeq0$.
Moreover, by \cite[2.2]{chw-v} we have
$\bbH_\cdh^{m}(R,HC^{(i)})\cong \bbH_\cdh^{2i+m}(R,\Omega^{\le i})$,
so the long exact sequence becomes
\begin{equation}\label{seq:key}
\cdots HC_n^{(i-1)}(R) \to \bbH_\cdh^{2i-n-2}(R,\Omega^{< i})
\to \tK_n^{(i)}(R)	\to HC_{n-1}^{(i-1)}(R) \cdots.
\end{equation}

The general picture is given by the following proposition.

\begin{prop}\label{weight1}
Let $R=R_0\oplus R_1 \oplus\cdots$ be as in Theorem \ref{thm:A}.
Then $\tK_n^{(0)}(R)=0$ for all $n$. For $n\le0$,
or for $n>0$ and $i\ge n+2$, we have
\[ 
\tK_n^{(i)}(R) \cong \bbH_\cdh^{2i-n-2}(R,\Omega^{<i}),
\quad \text{except for\quad}(n,i)=(0,1),
\]
In the exceptional case, $\tK_0^{(1)}(R)=\Pic(R) = R^+/R_{red}$.
\end{prop}

\begin{proof}
The group $HC_n(R)$ vanishes for $n < 0$ and is $R$ for $n=0$.
Similarly, $HC_n^{(i)}(R)$ vanishes for $i>n>0$.
The  proposition now follows from \eqref{seq:key} and the fact that
$H^0_{cdh}(R,\cO) = R^+$ by \cite[2.5]{chwnk}.
\end{proof}

To go further, it is useful to invoke the following trick,
using the standard $\mathbb{A}^1$-contraction of a cone to its vertex.

\begin{trick}\label{trick}
If $R$ is a positively graded algebra, there is an algebra map
$\nu:R\to R[t]$ sending $r\in R_n$ to $rt^n$. If $F$ is a functor on
algebras, then the composition of $\nu$ with evaluation at $t=0$
factors as $R\to R_0\to R$, so $F(R)\map{\nu} F(R[t])\map{t=0} F(R)$
is zero on the kernel $\widetilde{F}(R)$ of $F(R)\to F(R_0)$. Similarly,
the composition of $\nu$ with evaluation at $t=1$ is the identity. That is,
$\nu$ maps $\widetilde{F}(R)$ isomorphically onto a summand of $NF(R)$,
and $\widetilde{F}(R)$ is in the image of the map $(t=1):NF(R)\to F(R)$.
\end{trick}

The following technical result is crucial for our calculations; it
asserts that many SBI sequences 
decompose into split short exact sequences. We write $\cF_{HH}$ and
$\cF_{HC}$ for the homotopy fibers of $HH(R)\to\bbH_\cdh(R,HH)$ and
$HC(R)\to\bbH_\cdh(R,HC)$, respectively.
Then we have distinguished cohomological triangles
\begin{gather*}
\cF_{HC}[-1] \smap{S}  \cF_{HC}[1]   \smap{B} \cF_{HH}     \smap{I}\cF_{HC},
\\ \bbH_\cdh(R,HC)[-1] 
\smap{S}\bbH_\cdh(R,HC)[1] \smap{B} \bbH_\cdh(R,HH)\smap{I}\bbH_\cdh(R,HC).
\end{gather*}

\begin{lem}\label{SBI-FHC}
If $R=R_0\oplus R_1\oplus\cdots$ is a graded algebra then for each $m$ the map
$\pi_m\cF_{HC}(R)\map{S}\pi_{m-2}\cF_{HC}(R)$ is zero, and there
is a split short exact sequence:
\[
0 \to \pi_{m-1}\cF_{HC}(R) \map{B} \pi_m\cF_{HH}(R)
	\map{I} \pi_m\cF_{HC}(R) \to 0.
\]
Similarly, there are split short exact sequences:
\[
0\to \widetilde \bbH_\cdh^{m+1}(R,HC) \map{B}
 	\widetilde \bbH_\cdh^m(R,HH)\map{I} \widetilde \bbH_\cdh^m(R,HC) \to 0.
\]
and
\begin{equation*}
0\to \widetilde{\bbH}_\cdh^{n-1}(R,\Omega^{<i}) \map{B}
\widetilde{H}_\cdh^{n-i}(R,\Omega^i) \map{I}
\widetilde{\bbH}_\cdh^n(R,\Omega^{\le i}) \to 0.
\end{equation*}
\end{lem}

\begin{proof}
The third sequence is obtained from the second one by taking the
$i^{th}$ component in the Hodge decomposition,
described in \cite[2.2]{chw-v}, and setting $n=2i+m$.
For the first two sequences to split, 
it suffices to show that $I$ is onto and split.

By \cite[2.4]{chw-v}, $\cF_{HH}(\k)=\cF_{HC}(\k)=0$, so
$\widetilde\cF_{HH}=\cF_{HH}$ and $\widetilde\cF_{HC}=\cF_{HC}$.
By the standard trick \ref{trick}, it suffices to show that the maps
$N\pi_m\cF_{HH}(R)\to N\pi_m\cF_{HC}(R)$ and
$N\bbH_\cdh^m(R,HH)\to N\bbH_\cdh^m(R,HC)$ are onto and split.
But they are split surjections, as is evident from the
respective decompositions of their terms
in \cite[3.2]{chwnk} and \cite[2.2]{chwnk}; $\bbH_\cdh(R,NHH^{(i)})\simeq \bbH_\cdh(R,NHC^{(i)})\oplus \bbH_\cdh(R,NHC^{(i-1)})$ and $N\cF_{HH}^{(i)}(R)\simeq N\cF_{HC}^{(i)}(R)\oplus N\cF_{HC}^{(i-1)}(R)$.
\end{proof}

Splicing the final sequences of Lemma \ref{SBI-FHC} together, we see that
the de Rham complexes are exact in $cdh$-cohomology:

\begin{prop}\label{SBI-Homega}
The following sequences are exact:
\addtocounter{equation}{-1}
\begin{subequations}
\begin{align} \label{dRH0}
0\to k \to R^+ & \map{d} \widetilde{H}_\cdh^0(R,\Omega^1) \map{d}
\widetilde{H}_\cdh^0(R,\Omega^2) \to\cdots\\ \label{dRHn}
0\to \Hcdh^m(R,\cO) & \map{d} \Hcdh^m(R,\Omega^1) \map{d}
  	\Hcdh^m(R,\Omega^2) \to \cdots, \qquad m>0.
\end{align}
\end{subequations}
Note that the first complex is the $cdh$ {\em reduced} de Rham complex.
\end{prop}

An analogous exact sequence
$$
\cdots \to \pi_{m-1}\cF_{HH}(R)\smap{d}
\pi_m\cF_{HH}(R)\smap{d}\pi_{m+1}\cF_{HH}(R) \to \cdots
$$
is obtained by splicing the other sequences in \ref{SBI-FHC}.
Using the interpretation of their Hodge components, described in
\cite[3.4]{chwnk}, produces two more exact sequences:

\begin{prop}
The following sequences are exact:
\addtocounter{equation}{-1}
\begin{subequations}
\begin{align}\label{dRnil}
0\to \nil(R) \to & \tors\Omega^1_R \to \tors\Omega^2_R \to
\tors\Omega^3_R \to\cdots \\ \label{dROcdh}
0\to (R^+/R) \to & \Omega^1_\cdh(R)/\Omega^1_R \to
\Omega^2_\cdh(R)/\Omega^2_R \to \cdots.
\end{align}
\end{subequations}
\end{prop}

Here we have used the following notation
\begin{gather}\Omega^i_\cdh(R)=\Hcdh^0(R,\Omega^i)\label{Omegacdh}\\
\tors\Omega^i_R=\ker(\Omega^i_R\to\Omega^i_\cdh(R))\label{tors}
\end{gather}
If $R$ is reduced then $\tors\Omega^i_R$ is the usual torsion submodule,
by \cite[5.6.1]{chwnk}.

We can now make the calculations necessary
to deduce Theorem \ref{thm:A}.

\begin{thm}\label{thm:hypercohom}
Let $R=R_0\oplus R_1\oplus\cdots$ be a 
graded algebra, finitely generated over a field $k$ of characteristic $0$.
Assume $R_0$ is local artinian with residue field $k$.
Then we have
\[ \bbH_\cdh^{q+i}(R,\Omega^{\le i}) = \begin{cases}
H_{dR}^{q+i}(k),& q<0; \\ \coker\bigl\{
H_\cdh^q(R,\Omega^{i-1}) \map{d} H_\cdh^q(R,\Omega^i)\bigr\},& q\ge0; \\
0, & q\ge \dim(R).
\end{cases}\]
\end{thm}

\begin{proof}
The Cartan-Eilenberg spectral sequence for $\Omega^{\le i}$ is
\[ {}^I{\!}E_1^{p,q}=
\Hcdh^q(R,\Omega^p)\Longrightarrow \bbH_\cdh^{p+q}(R,\Omega^{\le i})
\qquad (0\le p\le i,~ q\ge0). \]
(See \cite[5.7.9]{WeibelHA94}.) Since
$\Hcdh^0(R,\Omega^p)=\Omega^p_{k}\oplus\widetilde{H}_\cdh^0(R,\Omega^p)$,
the row $q=0$ is the brutal truncation of the direct sum of the
de Rham complex of $k$ over $\Q$ and the complex \eqref{dRH0},
which is acylic by Proposition \ref{SBI-Homega}.
Since $\Hcdh^q(R,\Omega^p)=\widetilde{H}_\cdh^q(R,\Omega^p)$
for $q>0$, the other rows on the $E_1$-page are the truncations of the
complex \eqref{dRHn}, which is also acylic by \ref{SBI-Homega}.
Hence the spectral sequence degenerates at $E_2$, yielding the calculation.
Note that the last possible nonzero group is
$\bbH_\cdh^{i+\dim\,R-1}(R,\Omega^{\le i})=\Hcdh^{\dim\,R-1}(R,\Omega^i)$
by the cohomological bound in \cite[2.6]{chw-v}.
\end{proof}

\begin{proof}[Proof of Theorem \ref{thm:A}]
Simply plug the calculations of Theorem \ref{thm:hypercohom} into
those of Proposition \ref{weight1} to get the asserted result.
\end{proof}

We conclude the section with a calculation of the higher $K$-theory
of $R$ in terms of K\"ahler differentials, the cyclic homology of $R$
and the $cdh$-cohomology of $\Spec(R)$.
In the next section, we will reinterpret Theorems \ref{Kni}
and \ref{thm:Kbis} in terms of the Zariski cohomology of $X=\Proj(R)$.

\begin{thm}\label{Kni}
Let $R=R_0\oplus R_1\oplus\cdots$ be a finitely generated 
graded algebra over a field $k$ of characteristic~0. Assume $R_0$ is local
artinian with residue field $k$. Then for $n\ge1$ we have:
\smallskip

(a) $K_n^{(i)}(R) \cong HC_{n-1}^{(i-1)}(R)$ whenever $0<i<n$;
\smallskip

(b) $\tK_n^{(n)}(R)\cong\tors\Omega^{n-1}_{R}/d\tors\Omega^{n-2}_R.$
In particular, $\tK_1^{(1)}(R) \cong \nil(R)$ and
$$\tK_2^{(2)}(R) \cong \tors\Omega^1_R/d\nil (R).$$

(c) $K_n^{(n+1)}(R) \cong \coker\bigl\{
\Omega_\cdh^{n-1}(R) \map{d} \Omega_\cdh^n(R)/\Omega^{n}_R \bigr\}.$

(d) $K_n^{(i)}(R) \cong \coker\bigl\{
H_\cdh^{i-(n+1)}(R,\Omega^{i-2}) \map{d} H_\cdh^{i-(n+1)}(R,\Omega^{i-1})
\bigr\}$ when $i\ge n+2$.
\end{thm}

\begin{proof}
By Theorem \ref{thm:hypercohom}, we have
$\widetilde\bbH_\cdh^{m}(R,\Omega^{\le i})=0$ whenever $m<i$
({\it i.e.}, $q<0$).
Substituting this into \eqref{seq:key} gives assertion (a), 
because $HC_n^{(i)}(k) \map{\simeq} H^{2i-n}_{dR}(k/\Q)$ also holds.
Taking $m=i$, it also gives exactness of the top row in the diagram:
$$\minCDarrowwidth25pt\begin{CD}
0 @>>> \tK_n^{(n)}(R) @>>> \widetilde{HC}_{n-1}^{(n-1)}(R) @>>>
	\widetilde\Omega_\cdh^{n-1}(R)/d\Omega_\cdh^{n-2}(R) \\
@. @VVV @V{B}V{\text{into}}V @V{B}V{\text{into}}V \\
0 @>>> \tors\Omega^n_R @>>> \Omega^{n}_R/\Omega^n_{\k}
	@>>> \Omega_\cdh^n(R)/\Omega^n_{\k} \\
@. @VdVV @VdVV @VdVV \\
0@>>> \tors\Omega^{n+1}_R @>>> \Omega^{n+1}_R/\Omega^{n+1}_\k
@>>> \Omega_\cdh^{n+1}(R)/\Omega^{n+1}_\k.
\end{CD}$$
The other two rows are exact by definition, see \eqref{tors}.
The two right columns are exact by \cite[9.9.1]{WeibelHA94} and
\eqref{dRH0}, respectively. By a diagram chase, $\tK_n^{(n)}(R)$ is
the kernel of $\tors\Omega^n_R\to\tors\Omega^{n+1}_R$.
Part (b) now follows from \eqref{dRnil}.
Part (c) is immediate from \eqref{seq:key},
given the following information: $\bbH_\cdh^n(R,\Omega^{\le n})$ is the
cokernel of $d:\Omega_\cdh^{n-1}(R)\to\Omega_\cdh^n(R)$ by
Theorem \ref{thm:hypercohom}, $HC_n^{(n)}(R)=\Omega^n_R/d\Omega^n_R$
and $HC_{n-1}^{(n)}(R)=0$.
Part (d) follows from Theorem \ref{thm:hypercohom} and the formula
$\tK_n^{(i)}(R) \cong \bbH_\cdh^{2i-n-2}(R,\Omega^{<i})$ for $i\ge n+2$,
which is Proposition \ref{weight1}.
\end{proof}

\begin{cor}
If $i>n$ and $(n,i) \ne (0,1)$, the map $K_n^{(i)}(R)\to K_n^{(i)}(R^+)$ is an isomorphism.
\end{cor}

\begin{proof} For $n \geq 1$, it follows from Theorem \ref{Kni}, and
  for $n = 0$, it follows from Proposition \ref{weight1}.
\end{proof}

If the dimension of $R$ is $2$ (for example, if $R$ is the
cone over a projective curve), then the calculations of
Theorem \ref{Kni} apply to compute the higher
$K$-groups of $R$, but here the more dominant role is played by
K\"ahler differentials.
As in \eqref{Omegacdh}, we write $\Omega^i_\cdh(R)$ for $\Hcdh^0(R,\Omega^i)$.

\begin{thm}\label{thm:Kbis}
Assume $\dim(R)=2$ and that $R$ is reduced.
Then we have:
\begin{enumerate}
\item
$K_1(R)=\k^\times \oplus K_1^{(2)}(R)\oplus K_1^{(3)}(R)$ with
$K_1^{(i)}(R)=0$ for all $i\ge4$, with:
\begin{gather*}
K_1^{(2)}(R) \cong \Omega^1_\cdh(R)/(\Omega^1_R+d(R^+)),\quad\text{and}
\\
K_1^{(3)}(R) \cong \bbH_\cdh^3(R,\Omega^{\le2})
\cong \coker\bigl\{
\Hcdh^1(R,\Omega^1) \map{d}\Hcdh^1(R,\Omega^2)\bigr\};
\\
\end{gather*}
\item $K_2(R) \cong K_2(\k)\oplus \tors\Omega^1_R \oplus K_2^{(3)}(R)
\oplus K_2^{(4)}(R)$ with
\begin{gather*}
K_2^{(3)}(R) \cong \Omega^2_\cdh(R)/(\Omega^2_R+d\Omega^1_\cdh(R))
\quad\text{and} \\
K_2^{(4)}(R) \cong \coker\bigl\{
H_\cdh^1(R,\Omega^2) \map{d} H_\cdh^1(R,\Omega^3) \bigr\};
\end{gather*}
\item For all $n\ge3$, $K_n(R)\cong K_n(\k)\oplus
\bigoplus_{i=2}^{n+2} \tK_n^{(i)}(R)$, where
$$\tK_n^{(i)}(R) = \begin{cases} \widetilde{HC}_{n-1}^{(i-1)}(R), & i<n, \\
\tors\Omega^{n-1}_{R}/d\tors\Omega^{n-2}_R, & i=n, \\
\coker\bigl\{ \Omega_\cdh^{n-1}(R) \map{d} \Omega_\cdh^n(R)/\Omega^{n}_R
	\bigr\}, & i=n+1, \\
\coker\bigl\{ H_\cdh^1(R,\Omega^{n}) \map{d} H_\cdh^1(R,\Omega^{n+1})
	\bigr\}, & i=n+2.
\end{cases}$$
\end{enumerate}
\end{thm}

\begin{proof}
For $n=1$ we see from Remark \ref{rem:R+}
that $\tK_1^{(1)}(R)=\nil(R)=0$, and from
Theorem \ref{Kni}(c)
that $K_1^{(2)}(R)$ is the cokernel of $d:R^+ \to\Omega_\cdh^1(R)/\Omega^1_R$.
Since $R^+\to \Omega_\cdh^1(R)$ factors through $\Omega^1_{R^+}$,
the description of $K_1^{(2)}(R)$ follows.
From \eqref{seq:key}, we have
$K_1^{(3)}(R)\cong \bbH_\cdh^{3}(R,\Omega^{\le2})$, which is
described by \ref{thm:hypercohom}, and
$K_1^{(i)}(R)=\bbH_\cdh^{2i-3}(R,\Omega^{<i})$ for $i\ge4$,
which vanishes because
$\bbH_\cdh^m(R,\Omega^{<i})=0$ for $m\ge1+i$ by Theorem \ref{thm:hypercohom}.

For $n\ge2$, $K_n^{(i)}(R)$ was described in
Proposition \ref{weight1} and Theorem \ref{Kni}.
\end{proof}

\begin{lem}\label{lem:Omega/dOmega}
Assume that $R=k\oplus R_1\oplus\cdots$ is graded and $\dim(R)=2$.
Then for all $i\ge2$:
\[
\Omega^i_{R/k}/d(\Omega^{i-1}_{R/k})\cong
\tors\Omega^i_{R/k}/d(\tors\Omega^{i-1}_{R/k}).
\]
\end{lem}

\begin{proof}
For $i\ge3$ the $R$-module $\Omega^i_{R/k}$ is torsion because
$\Omega_\cdh^i(R/k)=0$. For $i=2$ we simply chase the diagram
$$\minCDarrowwidth25pt\begin{CD}
\tors\Omega^1_{R/k} @>>> \tors\Omega^2_{R/k} @>>>\tors\Omega^3_{R/k}
	@>>>\tors\Omega^4_{R/k} \\
@VV{\text{into}}V @VV{\text{into}}V @| @| \\
\quad\Omega^1_{R/k} @>d>> \quad\Omega^2_{R/k} @>d>>
	\quad\Omega^3_{R/k} @>d>> \quad\Omega^4_{R/k},
\end{CD}$$
comparing the exact sequence for $\tors\Omega^*_{R/k}$,
analogous to \eqref{dRnil},
to the de Rham sequence for $\Omega^*_{R/k}$ (which is exact by \cite[9.9.3]{WeibelHA94}).
\end{proof}

\begin{prop}\label{number-curve}
If $k$ is algebraic over $\Q$ and $R=k\oplus R_1\oplus\cdots$
is seminormal of dimension~2, then: 

\noindent a) $K_1(R)\cong k^\times\oplus \Omega^1_\cdh(R)/\Omega^1_R$;

\noindent b) $K_2(R) \cong K_2(k)\oplus \tors\Omega^1_R$;

\noindent c) $K_n(R)\cong K_n(k)\oplus \widetilde{HC}_{n-1}(R)$,
\qquad $n\ge3$.
\end{prop}

\begin{proof}
These assertions are special cases of Theorem \ref{thm:Kbis}. Using
Lemma \ref{lem:Omega/dOmega} for $n\ge3$ we have
\[
\tK_n^{(n)}(R) \cong \tors\Omega^{n-1}_R/d\tors\Omega^{n-2}_R
	\cong \Omega^{n-1}_R/d\Omega^{n-2}_R = HC_{n-1}^{(n-1)}(R).
\]
By \eqref{dRH0}, $K_n^{(n+1)}(R)$ is a subquotient of
$\Omega^{n+1}_\cdh(R)$ and vanishes for $n\ge2$; by \eqref{dRHn},
$K_n^{(n+2)}(R)$ is a subgroup of $H_\cdh^1(R,\Omega^{n+2})$
and vanishes for $n\ge1$.
\end{proof}

We conclude this section with two classical examples for which
$\Spec(R)$ has a smooth affine $cdh$ cover,
so that $\Omega_\cdh^*$ is easy to determine.

\begin{ex}\label{cusp}
The cusp $R=k[t^2,t^3]$ has $R^+=k[t]$ and
$K_1^{(2)}(R)=\Omega_\cdh^1/d(R^+)=\Omega^1_k$ (cf.\ \cite[12.1]{Kruse}).
The computation of $K_n(R)$ for $n\ge2$ is also easily derived from
Theorem \ref{Kni}, and stated explicitly in \cite[6.7]{GRW}.
\end{ex}

\begin{ex}\label{skew}
The seminormal ring $R=k[x_1,x_2,y_1,y_2]/( \{ x_iy_j : 1\le i, j\le 2 \})$ is the
homogeneous coordinate ring of a pair of skew lines in $\bbP^3$.
Its normalization is $\tR=k[x_1,x_2]\times k[y_1,y_2]$, and
$\Spec(\tR)\to\Spec(R)$ is a $cdh$ cover. It is easy to see that
$H^1_\cdh(R,\Omega^i)=0$, and $\Omega^i_R\to\Omega^i_\cdh(R)$ is
onto for $i\ne0$. Applying Theorem \ref{thm:Kbis}, we see that
$K_0(R)=\Z$, $K_1(R)=k^\times$ and $K_{-1}(R)=0$.
This recovers a classic result of Murthy in \cite{Murthy}.
If $k$ is algebraic over $\Q$ then we also have $\tors\Omega^1_R\cong k^4$
(on the $x_idy_j$), $\tors\Omega^2_R\cong k^4$ (on the $dx_idy_j$)
and $\Omega^3_R=0$, so by Proposition \ref{number-curve} we have
\[
K_2(R)=K_2(k)\oplus k^4,\quad\text{while}\quad
\tK_n(R)=\widetilde{HC}_{n-1}(R)\quad\text{for all $n\ge3$.}
\]
\end{ex}

\section{Affine cones of smooth varieties}\label{sec:cones}

Let $X$ be a smooth projective variety in $\bbP^N$, and let
$R = k \oplus R_1 \oplus R_2 \oplus \cdots$ be the
associated homogeneous coordinate ring.
We will write $L$ for the
pullback to $X$ of the ample bundle $\cO(1)$ on $\bbP^N$, and if
$\cF$ is a quasi-coherent sheaf on $X$, we write
$\cF(t)$ for $\cF \oo_{\cO_X} L^t$.
In this section we compute the
$cdh$ cohomology of $\Spec(R)$ and use it to compute the $K$-theory of $R$,
via Proposition \ref{weight1}. 
The main result is the theorem below,
computing the non-positive $K$-groups of $R$. Later in this section, we
give partial calculations of the positive $K$-groups.

Recall from Proposition \ref{weight1} that $K_{-m}^{(0)}(R)=0$ for all
$m>0$ 
and $\tK_n^{(0)}(R)=0$ for $n\ge0$.
Thus we are interested in $K_{-m}^{(i+1)}(R)$ for $i\ge0$.

\begin{thm} \label{thm:main}
Let $X$ be a smooth projective variety in $\bbP^N$ with
homogeneous coordinate ring $R$. Then
\begin{gather*}
K_0^{(1)}(R) \cong R^+/R =
\bigoplus\nolimits_{t = 1}^\infty H^0(X,\cO_X(t))/R_t, \quad\text{and}\\
K_0^{(i+1)}(R) \cong
\bigoplus\nolimits_{t=1}^\infty H^{i}(X,\Omega_X^{i}(t)),
\quad\text{for all~} i\ge1.
\end{gather*}
For any $m>0$, and all $i\ge0$, we have:
\begin{equation*}
K_{-m}^{(i+1)}(R) \cong
\bigoplus\nolimits_{t=1}^\infty H^{m+i}(X,\Omega_X^{i}(t)).
\end{equation*}

If $k$ has finite transcendence degree over $\Q$ then each
vector space $K_0(R)/\Z$ and $K_{-m}(R)$ is finite-dimensional.
\end{thm}

A few parts of Theorem \ref{thm:main} are easy to prove.
The formula $K_0^{(1)}(R)= R^+/R$ is given in Proposition \ref{weight1}.
Since $\Spec(R) \setminus \{\frakm_R\}$ is regular, we see from Remark
\ref{rem:R+} that $R^+$ agrees with the normalization $\tR$ of $R$
in degrees $t>0$, and it is well known that
$\tR = \bigoplus_{t=0}^\infty H^0(X,\cO(t))$;
see  \cite[Theorem 7.16]{Iit} and
\cite[Ch VII, \S 2, Remark at the bottom of page 159]{ZS}.
This yields the first display.
The final assertion, when $\text{tr.\,deg.}(k/\Q)<\infty$,
follows from the fact that each $\Omega^i_X$ is a coherent sheaf;
for each $q>0$ the $H^q(X,\Omega^i_X(t))$ are finite-dimensional,
and only finitely many are nonzero,
by Serre's Theorem~B (\cite[III.5.2]{Hart}).

The proof of the rest of the theorem will be given in
Corollary \ref{K-m} and Proposition \ref{weight2},
building upon several intermediate results.

To compute the $cdh$ cohomology of $\Spec(R)$, we will use the blowup
$Y$ of $\Spec(R)$ at the origin (i.e., at $\frakm_R$).
The following description of $Y$ is well known.

\begin{lem}\label{bundle}
The exceptional fiber of $\pi:Y\to\Spec(R)$ is isomorphic to $X$ and there
is a projection $p:Y \to X$ identifying $Y$ with the geometric line bundle
$\Spec_X(\Sym(L))$ over $X$, with sheaf of sections $L^*$. Moreover,
 the inclusion of the exceptional fiber $X$ into $Y$ is the zero
 section of the bundle $p: Y \to X$.
\end{lem}

\begin{proof}
The exceptional fiber is $\Proj$ of the Rees algebra
$\bigoplus \frakm^i/\frakm^{i+1}$, which is just $R$, and $X=\Proj(R)$
by construction. For each $x\in R_1$, the affine open
$D_+(x)$ of $X$ is $\Spec(A)$, where $R[1/x]=A[x,1/x]$, and the
line bundle $L^n$ restricts to the $A$-submodule $x^nA$ of $R[1/x]$.

We now consider $Y=\Proj(R[\frakm t])$. For $x\in R_1$, and $xt\in R_1t$,
the affine open $D_+(xt)$ in $Y$ is $\Spec(B)$, where
$R[\frakm t][1/xt]=B[xt,1/xt]$.
The graded map $R\cong\oplus R_it^i\to R[\frakm t]$
induces a projection $Y\to X$ as well as an inclusion of $A[x]$ in $B$.
This is onto, since $B$ is generated by elements of the form
$rt^m/(xt)^m=(r/x^n)x^{n-m}$
for $r\in R_{n}$, $n\ge m$. Hence $B=A[x]$. This shows that $Y$ is the
geometric line bundle over $X$, associated to the locally free sheaf $L$
(see \cite[Ex.\ II.5.18]{Hart}).
\end{proof}

By \cite{chsw} and \cite{chw-v}, we have split exact sequences
\begin{multline}\label{seq:cdh}
0\to H^0_\cdh(R,\cF)\to H^0_\zar(Y,\cF)\oplus\cF(\k) \to H^0_\zar(X,\cF)\to0,\\
0\to  H^m_\cdh(R,\cF) \to H^m_\zar(Y,\cF) \to H^m_\zar(X,\cF)\to0,
\quad\text{ for } m>0,
\end{multline}
when $\cF$ is one of the $cdh$ sheaves $\cO$ or $\Omega^i$, or
a complex of $cdh$ sheaves of the form $\Omega^{\le i}$.
Thus the calculation of $H^*_\cdh(R,\cF)$
is reduced to the calculation of $H^*_\zar(Y,\cF)$.

\begin{lem}\label{MarksLemma}
We have $H^0_\cdh(R,\cO)=R^+$ and
$H^m_\cdh(R,\cO) = \bigoplus_{t=1}^\infty H^m(X,\cO_X(t))$ for $m>0$.
\end{lem}

\begin{proof}
Since $p$ is affine, $H_\zar^*(Y,\cO_Y)=H_\zar^*(X,p_*\cO_Y)$, and
$p_*\cO_Y=\Sym(L)$ by Lemma \ref{bundle}. Hence
$H^m(Y,\cO)=\bigoplus_{t=0}^\infty H^m_\zar(X,\cO(t))$ for all $m$;
if $m=0$, this equals $R^+$. Now apply  \eqref{seq:cdh}.
\end{proof}

From Proposition \ref{weight1} and \ref{MarksLemma} we deduce the case
$K^{(1)}_*$ of Theorem \ref{thm:main}. For comparison, recall that
$K_0^{(1)}(R)=\Pic(R)$, $K_1^{(1)}(R)=R^\times=k^\times$ and
$K_n^{(1)}(R)=0$ for all $n\ge2$ by Soul\'e \cite{Soule}.

\begin{cor}\label{K-m}
For $m>0$ we have
\[K_{-m}^{(1)}(R)=H^m_\cdh(R,\cO) = \bigoplus_{t=1}^\infty H^m(X,\cO_X(t)).\]
\end{cor}

\begin{rem}
This clarifies results of Srinivas in \cite[Thm.\ 3]{Sr1}, \cite{Sr2}
and Weibel \cite{WeibelNorm}, which observed (when $X$ is a curve) that the
right side of the display in Corollary \ref{K-m} is an obstruction to
the vanishing of $\tilde{K}_0(R)$ and $K_{-1}(R)$.
\end{rem}

There is an exact sequence
$0 \to p^*\Omega^1_X \to \Omega^1_Y \to \Omega^1_{Y/X} \to 0$
of sheaves on $Y$. The relative sheaf $\Omega^1_{Y/X}$ is the line
bundle $p^*L$, and so we deduce exact sequences for all $i\ge1$:
\[
0 \to p^*\Omega^i_X\to \Omega^i_Y \to p^*(\Omega^{i-1}_X\otimes L) \to0.
\]
Since $p_*p^*\cF=\cF\otimes\Sym(L)$,
applying $p_*$ yields (graded) exact sequences of sheaves on $X$
for all $i \geq 1$:
\begin{equation} \label{pOmegaY}
0 \to \Omega^i_X\oo\Sym(L) \to p_*\Omega^i_Y \to
\Omega^{i-1}_X\oo L\oo\Sym(L) \to 0.
\end{equation}

\begin{lem}\label{HzarY}
The sequence \eqref{pOmegaY} determines a graded split exact sequence
\[ 0 \to
\bigoplus_{t=0}^\infty H_\zar^*(X,\Omega^i_X(t))
\to H_\zar^*(Y,\Omega^i_Y) \to
\bigoplus_{t=1}^\infty H_\zar^*(X,\Omega^{i-1}_X(t)) \to 0
\]
for each $i \geq 1$.
The left-hand map is an isomorphism in degree $0$, and
in degrees $t \geq
1$, its splitting is a consequence of the fact that the
composition
\begin{equation} \label{NewLabel}
H_\zar^*(X,\Omega^{i}_X(t))\to H_\zar^*(Y,\Omega^{i}_Y)
\map{d} H_\zar^*(Y,\Omega^{i+1}_Y) \to H_\zar^*(X,\Omega^{i}_X(t))
\end{equation}
is an isomorphism.
\end{lem}

\begin{proof}
It follows from \eqref{pOmegaY} that we have a (graded) exact sequence
\[ \cdots\smap{\partial}
\bigoplus_{t=0}^\infty H_\zar^*(X,\Omega^i_X(t)) \smap{p^*}
H_\zar^*(Y,\Omega^i_Y) \to
\bigoplus_{t=1}^\infty H_\zar^*(X,\Omega^{i-1}_X(t)) \smap{\partial}\cdots
\]
Therefore, the assertion that \eqref{NewLabel} is an isomorphism
implies the first assertion.
Referring to the maps of \eqref{pOmegaY}, it suffices to show that
the composition
$$
\Omega_X^{i}\otimes\Sym(L) \to p_*\Omega^{i}_Y \map{d}
p_*\Omega^{i+1}_Y \to \Omega^{i}_X\otimes L\otimes\Sym(L)
$$
is the evident graded surjection, with kernel $\Omega_X^{i}$.
But, in the notation of the proof of Lemma \ref{bundle}, it suffices to look
on the affine $D_+(x)=\Spec(A)$ of $X$, and here this is the map
$\Omega^{*}_A\otimes_A A[x] \to \Omega^*_A\otimes_A\Omega^1_{A[x]/A}$
sending $\omega\otimes x^n$ to $\omega\otimes nx^{n-1}\, dx$.
\end{proof}

\begin{subex}\label{ex:dRt}
In particular,
$0\to H^0(X,\Omega^1_X(t))\to H^0(Y,\Omega^1_Y)_t\to R_t\to0$ is exact
for $t\ge1$, and the composition $R_t \map{d}H^0(Y,\Omega^1_Y)_t\to R_t$
is an isomorphism.
\end{subex}

\begin{cor}\label{split}
For $i\ge1$ and $m\ge1$ we have:
\begin{gather*}
\Omega_\cdh^i(R) \cong\Omega^i_k \oplus\bigoplus_{t=1}^\infty
H_\zar^0(X,\Omega^i(t))\oplus H_\zar^0(X,\Omega^{i-1}(t)); \\
H_\cdh^m(R,\Omega^i) \cong \bigoplus_{t=1}^\infty
H_\zar^m(X,\Omega^i(t))\oplus H_\zar^m(X,\Omega^{i-1}(t)).
\end{gather*}
The cokernel of $\Omega_\cdh^{i-1}(R)\smap{d}\Omega_\cdh^i(R)$ is
$\Omega^i_k/d\Omega^{i-1}_k \oplus
\bigoplus\nolimits_{t=1}^\infty H_\zar^0(X,\Omega^i_X(t))$,
and the cokernel of
$H_\cdh^m(R,\Omega^{i-1})\smap{d} H_\cdh^m(R,\Omega^{i})$
is the summand $\bigoplus_{t=1}^\infty H_\zar^m(X,\Omega^{i}_X(t))$.
\end{cor}

\begin{proof}
The first assertions follow from Lemma \ref{HzarY} and \eqref{seq:cdh}.
The cokernel assertions follow from this using \eqref{dRH0}, \eqref{dRHn}
and induction on $i$.
\end{proof}

We may now deduce the remaining cases of Theorem \ref{thm:main},
the main theorem of this section. Recall that $K_{-m}^{(1)}(R)$
is $\oplus_t H^m(X,\cO(t))$ by Corollary \ref{K-m}.

\begin{prop}\label{weight2}
For $i\ge 1$, we have
\[
K_{-m}^{(i+1)}(R) \cong \bbH_\cdh^{m+2i}(R,\Omega^{\le i}) \cong
\bigoplus_{t=1}^\infty H^{m+i}(X,\Omega^{i}_X(t)),\qquad m\ge0.
\]
\end{prop}

\begin{proof}
The first isomorphism is Proposition \ref{weight1}. 
The second isomorphism is established in Lemma \ref{HzarY},
using the isomorphism
$$
\bbH_\cdh^{m+2i}(R,\Omega^{\le i}) \cong \coker\bigl\{
H_\cdh^{m+i}(R,\Omega^{i-1}) \map{d} H_\cdh^{m+i}(R,\Omega^{i})\bigr\}
$$
of Theorem \ref{thm:hypercohom}.
\end{proof}

The proof of Theorem \ref{thm:main} is now complete.
We next deduce partial information about
the groups $K_n(R)$ for $n\ge1$.

\begin{prop}\label{K12}
Let $X$ be a smooth projective variety in $\bbP^N$ with
homogeneous coordinate ring $R$. 
Then for all $n\ge1$ we have graded isomorphisms:
\begin{gather*}
K_n^{(n+1)}(R) \cong
\coker\left\{
\Omega^n_R/d\Omega^{n-1}_R \to \bigoplus_{t=1}^\infty
H^0(X,\Omega^n_X(t)) \right\}; 
\\
K_n^{(i)}(R) \cong \bigoplus_{t=1}^\infty H^{i-n-1}(X,\Omega^{i-1}_X(t)),
\qquad i\ge n+2.
\end{gather*}
The graded decomposition of 
$K_n^{(n+1)}(R) = \bigoplus_{t=1}^\infty K_n^{(n+1)}(R)_t$ is:
\[
K_n^{(n+1)}(R)_t  \cong \coker\left\{
\left(\Omega^n_R/d\Omega^{n-1}_R\right)_t \to
H^0(X,\Omega^n_X(t)) \right\}.
\]
\end{prop}

\begin{proof}
By Theorem \ref{Kni}(c), 
\begin{align*}
K_n^{(n+1)}(R) & \cong
\Omega^n_\cdh(R)/(\Omega^n_{R}+d\Omega^{n-1}_\cdh(R)) \\
& = \coker\left(\Omega^n_R/d\Omega^{n-1}_R\to \Omega^n_\cdh(R)/d\Omega^{n-1}_\cdh(R)\right).
\end{align*}
Since $\widetilde H_\cdh^0(R,\Omega^n)=\Omega^n_\cdh(R)/\Omega^n_{\k}$ and
$\Omega^n_{\k}\subset\Omega^n_{R}$, we see from
Corollary \ref{split} that this is the cokernel of
$\Omega^n_R/d\Omega^{n-1}_R\to \bigoplus_{t=1}^\infty H^0(X,\Omega^1_X(t))$,
as claimed.
\end{proof}

\begin{rem}\label{rem:LvsO(1)}
When $X=\bbP^r$ is embedded in $\bbP^N$ as a subvariety of degree $d>r$,
our $L^t=\cO_X(t)$ agrees with $\cO_{\bbP^r}(d\!\cdot\!t)$, because it is the
pullback of $\cO_{\bbP^N}(t)$ to $X=\bbP^r$. Similarly, the terms
written as $\Omega_X^i(t)$ in Proposition \ref{K12} should be read as
$\Omega_{\bbP^r}^i\oo\cO_{\bbP^r}(d\!\cdot\!t)$.
\end{rem}

\section{Cones over smooth curves}\label{sec:curvecones}

In this section, we focus on the case when $X$ is a curve (i.e., a
smooth projective variety of dimension one, embedded in $\bbP^N$), and
apply the results of Sections \ref{sec:standard} and \ref{sec:cones}
in this case. Recall from Theorem \ref{thm:main} that $K_{-m}(R)=0$ for
$m>1$.

The simplest case is when $k$ is algebraic over $\Q$. In this case, we know
from Proposition \ref{number-curve} that $\tK_2(R)\cong\tors\Omega^1_R$
and if $n\ge3$ then $\tK_n(R)\cong\widetilde{HC}_{n-1}(R)$. It remains
to describe the situation when $-1\le n\le1$.

\begin{lem}\label{no-omega}
Suppose that $k$ is algebraic over $\Q$ and that $R$ is the homogeneous
coordinate ring of a smooth curve $X$ over $k$. Then
$K_{-1}(R)=\oplus_{t=1}^\infty H^1(X,\cO(t))$, $K_0(R)=\Z\oplus(R^+/R)$ and
$\tK_1(R)=\oplus_{t=1}^\infty H^0(X,\Omega^1_{X/k}(t))/\Omega^1_{R/k}$.
\end{lem}

\begin{proof}
By Theorem \ref{thm:main}, $K_0^{(i)}(R)=0$ for $i\ge3$ and
$K_{-1}^{(i)}(R)$ is zero for $i\ge2$, while
$K_{-1}^{(1)}(R)$ is the sum of the $H^1(X,\cO(t))$ by \ref{K-m}.
By Serre Duality, $K_0^{(2)}(R)$ is the sum of the
$H^1(X,\Omega^1_{X/\k}(t))=H^0(X,\cO_X(-t)){}^*$,
which are zero for all $t>0$.

The formula for $\tK_1(R)$ is immediate from
Propositions \ref{number-curve} and \ref{K12}.
\end{proof}

\begin{prop}\label{Kunneth-Kni}
Suppose that $R$ is the homogeneous coordinate ring of a smooth curve $X$
over a number field $F$ contained in $k$. Then for $R_k=R\oo_Fk$:
\smallskip

(a) For $i<n$ we have $\tK_n^{(i)}(R_k)\cong
\oplus_{p=0}^i\ \Omega^p_k\oo_F\tK_{n-p}^{(i-p)}(R)$.
\smallskip

(b) For all $n\ge2$, $\tK_n^{(n)}(R_k) \cong \oplus_{p=0}^{n-2}\
\Omega^{p}_k \oo_F \tK_{n-p}^{(n-p)}(R)$.
\smallskip

(c) For all $n\ge1$,
$
K_n^{(n+1)}(R_k) \cong \Omega^{n-1}_k \oo_F K_1^{(2)}(R).
$
\smallskip

(d) For all $n\ge0$, $K_n^{(n+2)}(R_k)\cong
\Omega^{n+1}_k\oo_F K_{-1}(R) \cong \Omega^{n+1}_k\oo_k K_{-1}(R_k)$.
\end{prop}

\begin{proof}
Write $\oo$ for $\oo_F$.
Part (a) is immediate from Theorem \ref{Kni}(a) and Kassel's base change
formula $\widetilde{HC}_*(R_k)\cong\Omega^*_k\oo\widetilde{HC}_*(R)$.
(See \cite[(3.2)]{Kassel}.)

For (b), recall that
$\tK_n^{(n)}(R_k)\cong\tors\Omega^{n-1}_{R_k}/d\tors\Omega^{n-2}_{R_k}$
by Theorem \ref{Kni}(b).  By the K\"unneth formula,
$\tors\Omega^n_{R_k}=\oplus_{p+q=n}\Omega^p_k\oo\tors\Omega^q_R$.
Filtering by $p\ge0$ yields a 2-diagonal spectral sequence computing the
kernel and cokernel of $d:\tors\Omega^{n-1}_{R_k}\to\tors\Omega^n_{R_k}$,
with $E_0^{p,-p}=\Omega^p_k\oo\tors\Omega^{n-p}_R$ and
$E_0^{p,-1-p}=\Omega^p_k\oo\tors\Omega^{n-p-1}_R$. By \eqref{dRnil},
we have $E_1^{p,-p}=\Omega^p_k\oo\tK_{n+1}^{(n+1)}(R)$ and
$E_1^{p,-1-p}=\Omega^p_k\oo d\,\tors\Omega^{n-p-2}_R$. Given
$\alpha$ in $\Omega^p_k$ and $d\tau$ in $d\,\tors\Omega^{n-p-2}_R$,
$d(\alpha\oo d\tau)=d\alpha\oo d\tau=d(d\alpha\oo\tau)$ in
$\tors\Omega^n_{R_k}$, which shows that $d^1=0$ and establishes (b).

By the K\"unneth formula and Proposition \ref{K12},
$K_n^{(n+1)}(R_k)$ is the direct sum over $p+q=n$ of the cokernels
of the maps
\[ \Omega^p_k\oo\Omega^q_R \to \Omega^p_k\oo H^0(Y,\Omega^q_Y)
 \to \Omega^p_k\oo\oplus_t H^0(X,\Omega^q_X(t)).
\]
For $q=0$, the composite is the identity map of $\Omega^n_k\oo R$.
For $q=1$, the composite is $\Omega^{n-1}_k$ tensored with the map
$\Omega^1_R\to\oplus_t H^0(X,\Omega^1_X(t))$ defining $K_1^{(2)}(R)$.
For $q\ge2$, the right side is zero.
This establishes part (c).

Since $K_{-1}(R)\cong H^1(X,\cO(t))$, part (d) is just
Proposition \ref{K12}, together with the K\"unneth formula that
$H^1(X_k,\Omega^n_{X_k}(t))$ is the direct sum of
$\Omega^n_k\oo H^1(X,\cO(t))$ and $\Omega^{n-1}_k\oo H^1(X,\Omega^1_X(t))$,
which is zero for $t>0$ by Serre Duality.
\end{proof}

When $k/\Q$ is transcendental, we will use a variant of
the arithmetic Gauss-Manin connection
$H^1_{dR}(X/k)\to\Omega^1_k\oo H^1_{dR}(X/k)$,
or rather its ($k$-linear) filtered piece
\[
\nabla:H^0(X,\Omega^1_{X/k})\to\Omega^1_k\oo H^1(X,\cO_X)
\]
as described in \cite[Thm.\,2]{KatzOda} and \cite[3.2]{LS}.
When $k=\C$, this can be interpreted in terms of the Hodge filtration as
a map $H^{1,0}(X,\C)\to \Omega^1_{\C/k}\oo H^{0,1}(X,\C)$.

It is known (see \cite{KatzOda}) that $\nabla$ is the cohomology
boundary map associated to the fundamental short exact sequence
$0\to \Omega^1_k\oo\cO_X\to\Omega^1_X\to\Omega^1_{X/k}\to0$.
Twisting this short exact sequence by $\cO(t)$ yields a twisted version
$\nabla_t:H^0(X,\Omega^1_{X/k}(t))\to\Omega^1_k\oo H^1(X,\cO(t))$.
We see from Lemma \ref{HzarY} that the direct sum of the $\nabla_t$ is
a component of the cohomology boundary map associated to
$0\to\Omega^1_k\oo\cO_Y\to\Omega^1_Y\to\Omega^1_{Y/k}\to0$;
it follows that $\oplus\nabla_t$ is $R$-linear.

Since $\Omega^2_{X/k}=0$,
we have fundamental exact sequences for each $i$:
\begin{equation}\label{fund_seq}
0\to\Omega^i_{\k}\oo\cO_X(t)\to\Omega^i_{X}(t)\to
\Omega^{i-1}_{\k}\oo\Omega^1_{X/\k}(t)\to 0.
\end{equation}
The cohomology boundary maps are the $k$-linear homomorphisms
\[
\Omega^{i-1}_k\oo H^0(X,\Omega^1_{X/k}(t))\map{\nabla_t}
\Omega^i_k\oo H^1(X,\cO(t)).
\]
The sum of the $\nabla_t$ is again $R$-linear, as the sum of the
sequences \eqref{fund_seq} is $R$-linear.
Alternatively, we can use the fact that the
arithmetic Gauss-Manin connection can be extended via the usual formula
$\nabla_t(\omega\oo x)=d\omega\oo x + (-1)^{i-1} \omega\land\nabla_t(x)$,
and the first term vanishes because it is in a lower part of
the Hodge filtration.

\begin{lem} \label{lem:FundSeq}
If $X$ is a smooth curve and $i\ge1$, there is a graded exact sequence
of $R$-modules, the sum over $t>0$ of the exact sequences
\begin{multline*}
0 \to \Omega^i_{\k}\oo R_t \to H^0(X,\Omega^i_{X}(t))
\to \Omega^{i-1}_{\k}\oo H^0(X,\Omega^1_{X/k}(t)) \map{\nabla_t} \\
\Omega^i_{\k}\oo H^1(X,\cO_X(t))\to H^1(X,\Omega^i_{X}(t))
\to 0.
\end{multline*}
Moreover, we have the identity
\[
\nabla_t(\omega\otimes x)=\omega\land \nabla_t(x), \qquad \text{for~
  $\omega\in \Omega^{i-1}_\k$ and $x\in H^0(X,\Omega^1_{X/k}(t))$.}
\]
\end{lem}

\begin{proof}
This is just the cohomology exact sequence for \eqref{fund_seq},
together with Serre Duality, which says that
$H^1(X,\Omega^1_{X/\k}(t))=H^0(X,\cO_X(-t))=0$
for all $t>0$.

To prove that the boundary map is $\nabla$, let $\cU$ be a cover of
$X$ by affine open subschemes and consider the exact sequence of
\v{C}ech complexes associated to \eqref{fund_seq}. We have
$\check{C}(\cU,\Omega^i_{\k}\oo\cO_X(t))=
           \Omega^i_{\k}\oo\check{C}(\cU,\cO_X(t))$
and
\[
\check{C}(\cU,\Omega^{i-1}_{\k}\oo\Omega^1_X(t))=
\Omega^{i-1}_{\k}\oo\check{C}(\cU,\Omega^1_{X/\k}(t)).
\]
Let $\omega\in \Omega^{i-1}_\k$ and $x\in H^0(X,\Omega^1_{X/k}(t))=
H^0\check{C}(\cU, \Omega^1_{X/k}(t))$. If $y\in \check{C}(\cU,\Omega^1_X(t))$
maps to $x$, then $\delta(y)$ is in $\Omega^1_k\oo\check{C}(\cU,\cO(t))$
and represents $\nabla(x)$. Since $\omega\land y$ lifts $\omega\oo x$,
$\nabla(\omega\oo x)$ is the class of $\delta(\omega\land y)$
in $\Omega^i_\k\oo H^1(X,\Omega^1_{X/k})$.
Since $\omega$ is globally defined, we have
$\delta(\omega\land y)=\omega\land\delta(y)$.
\end{proof}

\begin{prop}\label{K02}
If $X$ is a smooth curve, we have graded exact sequences
\begin{multline*}
0\to K_1^{(2)}(R) \to
\frac{\oplus_t H^0(X,\Omega^1_{X/k}(t))}{\text{\rm image~}\Omega^1_{R/k}}
\map{\nabla} \Omega^1_k\oo \left(\oplus_tH^1(X,\cO(t))\right)
	 \to K_0^{(2)}(R) \to 0; \\
0\to K_{n+1}^{(n+2)}(R) \to
\frac{\Omega^n_{\k}\oo\left[\oplus_t H^0(X,\Omega_{X/\k}^1(t))\right]}
			{\text{\rm image~}\Omega^{n+1}_{R}}
\map{\nabla} \Omega^{n+1}_{\k} \oo\left(\oplus_t H^1(X,\cO_X(t)\right)
\\  \to K_n^{(n+2)}(R) \to0, \qquad n\ge1.
\end{multline*}
The direct sums are taken from $t=1$ to $\infty$.
\end{prop}

\begin{proof}
This follows from the exact sequence of Lemma \ref{lem:FundSeq}, using
the formulas $K_n^{(n+2)}(R)_t\cong H^1(X,\Omega^{n+1}_X(t))$ and
$K_{n+1}^{(n+2)}(R)_t\cong H^0(X,\Omega^{n+1}_X(t))/
\text{im}\left(\Omega^{n+1}_R\right)_t$ of Propositions
\ref{weight2} and \ref{K12}, once we observe that the first map
of Lemma \ref{lem:FundSeq} factors through $\Omega^{i}_R$.
This is because it is a quotient of
$\Omega^{i}_k\oo R\to \pi_*(\Omega^i_Y)=H^0(Y,\Omega^i_Y)$, which
factors as $\Omega^i_k\oo R\to \Omega^i_R\to\pi_*(\Omega^i_Y)$.
\end{proof}

\begin{ex}\label{ex:partial_vanish}
If $X$ is a curve definable over a number field contained in $k$,
then the Fundamental Sequence \eqref{fund_seq} (with $i=1$ and $t=0$) splits as
$\Omega^1_{X}\cong\Omega^1_{X/\k}\oplus\Omega^1_{\k}\oo\cO_X$, by the
K\"unneth formula. This implies that
$\Omega^i_{X}\cong\left(\Omega^{i-1}_k\oo\Omega^1_{X/\k}\right)
\oplus\left(\Omega^i_{\k}\oo\cO_X\right)$, so the Gauss-Manin connection
$\nabla$ of Lemma \ref{lem:FundSeq} vanishes and therefore:
\begin{gather*}
K_n^{(n+1)}(R)=
\frac{\Omega^{n-1}_{\k}\oo\left[\oplus_t H^0(X,\Omega_{X/\k}^1(t))\right]}
			{\text{\rm image~}\Omega^{n}_{R}},
\quad n\ge1;\\
K_n^{(n+2)}(R)=\Omega^{n+1}_{\k}\oo
\left[\oplus_t H^1(X,\cO_X(t))\right]
\cong \Omega^{n+1}_{\k}\oo K_{-1}(R), \quad n\ge0.
\end{gather*}
\noindent Of course, the formula for $K_n^{(n+1)}(R)$ reduces to that of
Proposition \ref{Kunneth-Kni}(c).

The formula for $K_0^{(2)}(R)$ clarifies the examples given by
Srinivas in \cite{Sr1}. There it was shown that if $X$ is definable over
a number field, then $K_0(R)$ maps onto $\Omega^1_{\k}\oo
H^1(X,\cO_X(1))$ (see page 264).  From this Srinivas deduced that if
$\k=\mathbb C$ and $H^1(X,\cO_X(1))\ne0$ then $\tK_0(R)\ne0$.

The description of $K_0(R)=\Z\oplus K_0^{(2)}(R)$ in this
special case was independently discovered by Krishna and Srinivas
\cite{KSr2}.
\end{ex}

\begin{lem}\label{lem:degree1}
For any graded algebra $R=k\oplus R_1\oplus\cdots$, the
degree~1 part of $\Omega^i_R$ decomposes as
\[
(\Omega^i_R)_1 \cong (R_1\oo\Omega^i_k) \oplus (\Omega^{i-1}_k\oo R_1).
\]
The inclusions of $R_1\oo\Omega^i_k$ and $\Omega^{i-1}_k\oo R_1$ are given by
$r\oo\omega\mapsto r\omega$ and
$\omega\otimes r\mapsto \omega\land dr$, respectively.
\end{lem}

\begin{proof}
We may suppose for simplicity that $N=\dim(R_1)$ is finite, so that
the polynomial ring $S=k[x_1,\dots,x_N]$ maps to $R$, and
$S\to R$ is an isomorphism in degree~1.
For every subfield $\ell$ of $k$, $\Omega^1_{R/\ell}$ is the cokernel of
the Hochschild boundary $R^{\oo3}\to R\oo_\ell R$; thus
the map $\Omega^1_{S/\ell}\to\Omega^1_{R/\ell}$ is an isomorphism
in degree~1, and therefore so is $\Omega^i_{S/\ell}\to\Omega^i_{R/\ell}$.
Since $\Omega^1_{S}\cong(\Omega^1_{k}\oo S)\oplus\Omega^1_{S/k}$,
it is easy to check that the degree~1 part of $\Omega^i_{S}$ is
$(\Omega^i_k\oo S_1) \oplus  \Omega^{i-1}_k\oo S_1$, via the given formulas.
\end{proof}

\begin{thm}\label{nonvanish}
Let $X$ be a curve of genus $g$, embedded in $\bbP^N$ by a complete
linear system of degree $d>1$.
Assume that the twisted Gauss-Manin connection
$\nabla:H^0(X,\Omega^1_{X/k}(1)) \map{}\Omega^1_{\k}\oo H^1(X,\cO_X(1))$
is zero. Then $K_1^{(2)}(R)_1\cong \k^{d+g-1}\ne0$, and
\[
K_n^{(n+1)}(R)_{1}\cong\Omega^{n-1}_\k\otimes_\Q\k^{d+g-1} \qquad (n\ge 1).
\]
\noindent
In particular, $K_n^{(n+1)}(R)\ne0$ for all $n$ with
$1\le n<\text{tr.\,deg.}(k/\Q)$.
\end{thm}

\begin{proof}
By Proposition \ref{K12}, the degree~1 part of $K_n^{(n+1)}(R)$ is
\[
K_n^{(n+1)}(R)_1=
\coker((\Omega^n_R/d\Omega^{n-1}_R)_1\to H^0(X,\Omega^n_X(1))).
\]
By Lemmas \ref{lem:degree1} and \ref{lem:FundSeq}, and our hypothesis,
we have morphisms of exact sequences
\[
\xymatrix{\quad0\ar[r]&\Omega^n_k\otimes R_1\ar[d]^{\text{id}}\ar[r]
&(\Omega^n_R)_1\ar[r]^{\omega\land dr\mapsto \omega\oo r}\ar[d]&
\Omega^{n-1}_k\oo R_1\ar[r]\ar[d]^{1\oo d}    	& 0\quad\\
\quad0\ar[r]&\Omega^n_k\otimes R_1\ar[d]^{\text{id}}\ar[r]&
H^0(Y,\Omega^n_Y)_1\ar[d]\ar[r]
&\Omega^{n-1}_k\oo H^0(Y,\Omega^1_{Y/k})_1\ar[d]\ar[r]^-{\partial} &\quad\\
\quad0\ar[r]&\Omega^n_k\otimes R_1\ar[r]& H^0(X,\Omega^n_X(1))\ar[r]
&\Omega^{n-1}_k\otimes H^0(X,\Omega^1_{X/k}(1))\ar[r]^-{\nabla}&0\quad}
\]
where the bottom vertical maps are given in Lemma \ref{HzarY} as
the quotients by $dH^0(X,\Omega^{n-1}(1))$ and $\Omega^{n-1}_k\oo dR_1$.
It follows that the right vertical composite is zero.
Hence $K_n^{(n+1)}(R)_1$, which is the cokernel of the middle vertical
composite, is isomorphic to $\Omega^{n-1}_k\oo H^0(X,\Omega^1_{X/k}(1))$.
Finally, $\dim H^0(X,\Omega^1_{X/k}(1)))={d+g-1}$ by Riemann-Roch.
\end{proof}

\begin{ex}\label{ex:partial=0} Here are two cases in which the
hypotheses of Theorem \ref{nonvanish} above are satisfied:
\begin{enumerate}
\item[(a)] $X$ is embedded in $\bbP^N$ by a complete linear system
of degree $d\ge 2g-1$. In this case $\deg(\Omega^1_{X/k}(-1))<0$,
so $H^1(X,\cO(t))=0$ for all $t\ge1$ by Serre duality.
Theorem \ref{nonvanish} improves the result of Srinivas in \cite{SrCone}
that if $d\ge2g+1$ then $\tK_1(R)\ne0$.

\item[(b)] $X$ is definable over a number field contained in $k$.
\end{enumerate}
\end{ex}

\section{$K$-theory of the plane conic}\label{sec:conic}

We conclude with a classical example: $X$ is the plane conic with
homogeneous coordinate ring $R=k[x,y,z]/(z^2-xy)$. This curve is a
degree~2 embedding of $\bbP^1$ in $\bbP^2$; as pointed out in
Remark \ref{rem:LvsO(1)}, our line bundle $\cO_X(t)$ is the usual
$\cO_{\bbP^1}(2t)$.

Murthy observed long ago, in \cite[5.3]{Murthy}, that $K_0(R)=\Z$ and
$K_{-1}(R)=0$; this also follows from our Theorem \ref{thm:main}.
Srinivas proved in \cite{SrCone} that $\tK_1(R)$ surjects onto $k$.
Theorem \ref{thm:conic} below gives a complete calculation of $K_*(R)$,
or rather, $\tK_*(R)=K_*(R)/K_*(k)$.

\begin{lem}\label{lem:conic}
For $R=k[x,y,z]/(z^2-xy)$,
$\Omega^1_{R/k}$ is a torsionfree $R$-module,
\end{lem}

\begin{proof}
As $R$ is a normal complete intersection, a theorem of
Vasconcelos (\cite[2.4]{W80}) says that $\Omega^1_R$ is a torsionfree
$R$-module.  As such, 
it is a graded submodule of $\Omega^1_{R[1/x]}$. From the
factorization $\Spec(R[1/x])\!\to Y\!\to\Spec(R)$, we see that the graded
map $\Omega^1_{R/k}\to H^0(Y,\Omega^1_{Y/k})$ is an injection.
Since $R/k\smap{d}\Omega^1_{R/k}\to H^0(Y,\Omega^1_{Y/k})$ is an
injection with cokernel $\oplus_t H^0(X,\Omega^1_{X/k}(t))$ by
Lemma \ref{HzarY},
we are reduced to comparing the Hilbert functions of both sides.

It is easy to show that $\dim(R_t)=2t+1$ for all $t\ge0$. From the
resolution $0\to R(-2)\map{dF} R(-1)^3\to\Omega^1_{R/k}\to0$,
we compute that $\dim(\Omega^1_{R/k})_t$ is $3$ for $t=1$
and $4t$ for $t\ge2$.
By Riemann-Roch, we have $\dim H^0(X,\Omega^1_{X/k}(t))=2t-1$ for $t>0$.
By Lemma \ref{HzarY}, this yields:
$$
\dim\,H^0(Y,\Omega^1_{Y/k})_t=\dim\,H^0(X,\Omega^1_{X/k}(t))+\dim\,R_t
=(2t-1)+(2t+1)=4t.
$$
This shows that $(\Omega^1_{R/k})_t\cong R_t\oplus H^0(X,\Omega^1_{X/k}(t))$
when $t\ge2$, as desired.
\end{proof}
\begin{subremark}\label{rem:2-form}
Since $\Omega^1_R$ is torsionfree, the exact sequence \eqref{dRnil}
shows that $d:\tors\Omega^2_{R/k}\cong\Omega^3_{R/k}\cong k$.
In fact, the 2-form $\tau=z\,dx\land dy+2y\,dx\land dz$ has
$x\tau=y\tau=z\tau=0$ and $d\tau=dx\land dy\land dz$.
\end{subremark}

\begin{lem}\label{lem:HCconic}
For $R=\Q[x,y,z]/(z^2-xy)$ and $n\ge2$, $\widetilde{HC}_n^{(i)}(R)$ is
$\Q$ if $n=2i-2$ and zero otherwise. For $R_k=R\oo k$,
$\widetilde{HC}_n^{(i)}(R_k)$ is $\Omega^p_k$, where $p=2i-n-2$.
\end{lem}

\begin{proof}
The calculation of $HC_n^{(i)}(R)$ is taken from \cite[Thms.\,2--3]{Michler},
using the elementary calculation that $\Omega^3_R\cong\Q$ for $n>3$ and
exactness of the augmented Poincar\'e complex $\Q\to\Omega^*_R$ for $n=2,3$.
The second sentence follows using the base change formula of
\cite[(3.2)]{Kassel}.
\end{proof}

\begin{thm}\label{thm:conic}
For $R_k=k[x,y,z]/(z^2-xy)$ and all $n$, we have
\[
\tK_n(R_k) \cong \Omega^{n-1}_k\oplus\Omega^{n-3}_k\oplus\Omega^{n-5}_k
\oplus\cdots.
\]
In particular, $K_1(R_k)\cong K_1(k)\oplus k\quad\text{and}\quad
K_2(R_k)\cong K_2(k)\oplus\Omega^1_k.$
\end{thm}

\begin{proof}
By Proposition \ref{Kunneth-Kni}(a) and Lemma \ref{lem:HCconic},
we see that $\tK_n^{(n-j)}(R)$ is $\Omega^{n-2j-3}_k$ for all $j>0$.
By Theorem \ref{thm:Kbis} and Remark \ref{rem:2-form}, we have
$\tK_3^{(3)}(R_{\Q})\cong k$ and $\tK_n^{(n)}(R_{\Q})=0$ for $n\ne3$.
By Proposition \ref{Kunneth-Kni}(b) this implies that
$\tK_n^{(n)}(R_{k})\cong\Omega^{n-3}_k$ for all $n\ne3$.
By Proposition \ref{K02} and Lemma \ref{lem:conic},
we have $K_1^{(2)}(R_k)=k$. By Proposition \ref{Kunneth-Kni},
this implies that $K_n^{(n+1)}(R)\cong\Omega^{n-1}_k$
for all $n\ge1$. Finally, by Proposition \ref{K12}
we have $K_n^{(n+2)}(R_k)_t=H^1(X_k,\Omega_{X_k}^{n+1}(t))$,
which vanishes for all $n,t\ge1$ as it is the sum of
$\Omega^{n}_k\oo H^1(X,\Omega^1_X(t))$, which vanishes by Serre Duality,
and $\Omega^{n+1}_k\oo H^1(X,\cO_X(t))$, which vanishes as $X=\bbP^1$.
\end{proof}

\begin{subremark}
When $k$ is algebraic over $\Q$, the formulas in Theorem \ref{thm:conic}
reduce to: $\tK_n(R_k)=\Q$ for $n\ge1$ odd, and $\tK_n(R_k)=0$ otherwise.
\end{subremark}

\subsection*{Acknowledgements}
The authors would like to thank James Lewis for pointing out the
relation to the arithmetic Gauss-Manin connection.

\end{document}